\documentclass[a4paper,12pt]{article}
\usepackage{amsmath,amsthm,amsfonts}
\usepackage{fullpage,comment,graphicx,color,epstopdf}
\usepackage{multicol}
\usepackage{epsfig}
\usepackage{epstopdf}

\usepackage{float}

\usepackage{color,graphicx,times}

\input{xy}

\xyoption{all}
\xyoption{matrix}

\theoremstyle{plain}
 
 \newtheorem{teo}{Theorem}
  \newtheorem{lem}[teo]{Lemma}

  \theoremstyle{remark}
  \newtheorem{rem}[teo]{Remark}
  
  \newtheorem{ex}[teo]{Example}
  \newtheorem{question}[teo]{Question}
  
\theoremstyle{definition}
  
    \newtheorem{defi}[teo]{Definition}

\def\t{\triangleleft}

\def\ot{\otimes}

\def\wt{\widetilde}

\def\flip{\mathrm{flip}}

\def\Id{\mathrm{Id}}
\def\id{\mathrm{Id}}
\def\F{\mathbb{F}}
\def\R{\mathbb{R}}
\def\N{\mathbb{N}}

\def\Z{\mathbb{Z}}

\def\Om{\Omega}
\def\ra{\overline}

\def\m1{^{ \hbox{\small{-}}1}}

\title{Universal  cocycle 
Invariants for singular knots and links}

\author{Marco  Farinati\thanks{Member of CONICET. Partially supported by
PIP 11220110100800CO, and UBACYT 20021030100481BA, mfarinat@dm.uba.ar.} 
\ and Juliana Garc\'ia Galofre\thanks{Partially supported by 
PIP 11220110100800CO and UBACYT 20021030100481BA,
jgarciag@dm.uba.ar }
}
\begin{document}
\maketitle
\begin{abstract}

Given a biquandle $(X, S)$, a function $\tau$ with certain compatibility and a pair of
{\em non commutative cocyles} $f,h:X \times X\to G$ with values in a non necessarily 
commutative group $G$,
 we  give an invariant for singular knots / links. 
Given $(X,S,\tau)$, we also define a universal group $U_{nc}^{fh}(X)$ and  universal 
functions governing all 2-cocycles in $X$, and exhibit examples of computations.  
When the target group is abelian, a notion of {\em abelian cocycle pair} is given
and the ``state sum'' is defined for singular knots/links.
Computations generalizing linking number for singular knots are given.
As for virtual knots,  a ``self-linking number''  may be defined 
for singular knots.
\end{abstract}



\section*{Introduction and preliminaries}

Following methods from \cite{FG2} and \cite{CES}-\cite{CEGS} we define two 
types of cocycles and corresponding invariants for singular knots.

A singular knot is a smooth map $f:S^1\rightarrow \R^3$ whose 
image eventually  has singularities: a finite number of double points, with transversal crossings.
A singular link is the union of singular knots, that is generically disjoint: the
intersection is finite and transversal. 
Singular links/knots may be considered  to be equivalence classes of planar 
singular knot diagrams under the
equivalence relation generated by  the
three (classical) Reidemeister moves and  the singular (RV and RIV) Reidemeister moves
depicted in Figure \ref{RM}. A singular crossing is a crossing where two strands are fused
 together.  An orientation of each circle induces an orientation
on each component of the link.
All links and knots considered
in this work will be oriented ones.

\begin{figure}[ht]
\[
\begin{array}{c}
\includegraphics[scale=0.1]{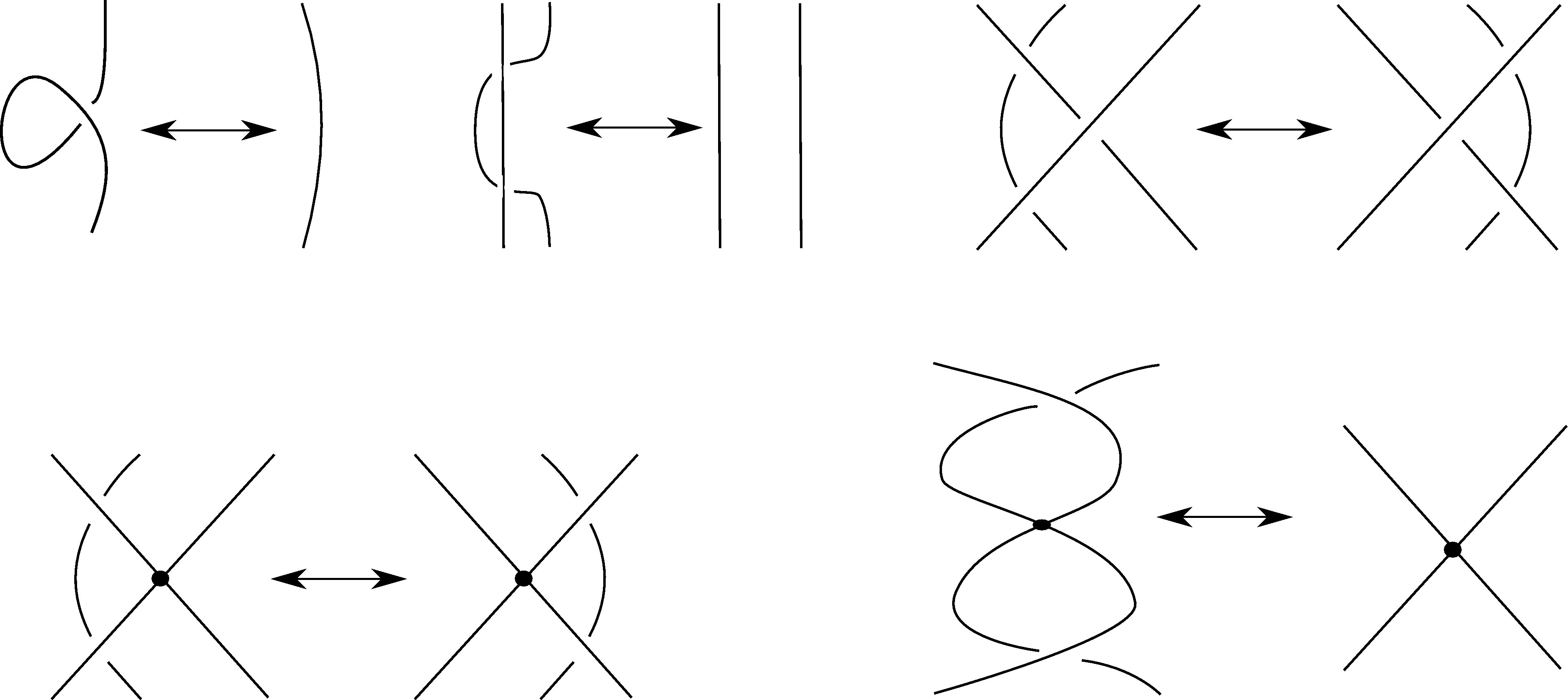}\\
   	\vspace{-.3in}
\end{array}
\]
\caption{Classical and singular Reidemeister moves RI,RII,RIII,RIV and RV.} 
\label{RM}
\end{figure}

After \cite{BEHY}, it is known that only 
 3 {\em oriented} singular Reidemeister moves are sufficient in order to generate them all. In 
\cite{BEHY} the basic singular movements are called $\Om4a$, $\Om4e$ and $\Om5a$ and correspond to our 
$oRIVb$, $oRIVa$ and (equivalent to) $oRV$ respectively (see figures \ref{oRIVb},
 \ref{oRIVa} and \ref{Om5a}). 

\begin{rem}
 In Figure \ref{RM}, the bottom left figure represents both of the following movements:
 \end{rem}

\begin{figure}[ht]   
\[
\begin{array}{c}
\includegraphics[scale=0.2]{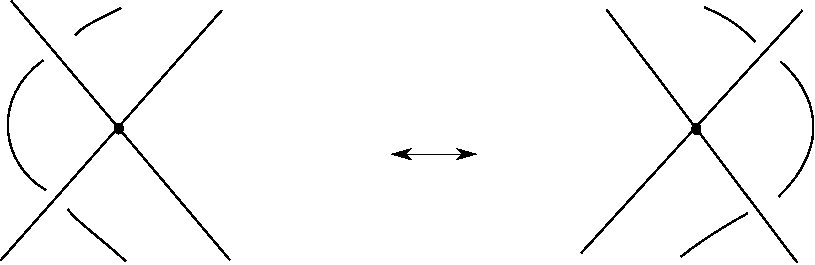}\\
   	\vspace{-.3in}
\end{array}\hspace{12mm}
\begin{array}{c}
\includegraphics[scale=0.2]{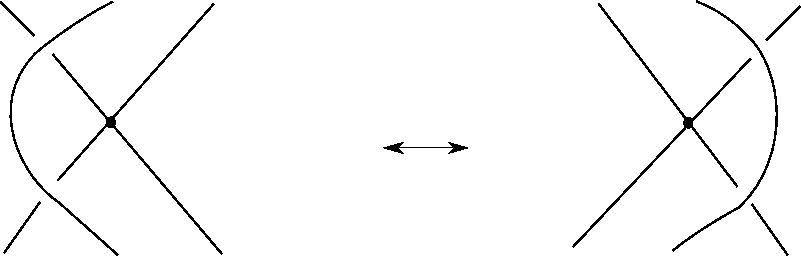}\\
   	\vspace{-.3in}
\end{array}
\]
 \caption{RIVa (on the left) and RIVb (on the right) moves. }
 \label{RVab}
 \end{figure}

\begin{figure}[H]
\[
  \xymatrix@-6pt{
  &\ar@{<-}[rdd]|(.3)\hole|(.51)\bullet&\ar@{->}[ldd]|(.3)\hole&\\
\ar@{->}@/^.83pc/[rrr]&&&\\  
&&&\\
} 
 \hspace{2mm}
\xymatrix{
\\
\equiv
}\hspace{2mm}
 \xymatrix@-6pt{
  &\ar@{<-}[rdd]|(.7)\hole&\ar@{->}[ldd]|(.49)\bullet|(.7)\hole&\\
\ar@{->}@/_.83pc/[rrr]&&&\\ &&&\\ 
} 
\hskip 1cm
  \xymatrix@-6pt{
  &\ar@{<-}[rdd]&\ar@{->}[ldd]|(.49)\bullet&\\
\ar@{->}@/^.8pc/[rrr]|(.42)\hole|(.55)\hole&&&\\  
&&&\\
}
 \hspace{2mm}
\xymatrix{
\\
\equiv
}\hspace{2mm}
 \xymatrix@-6pt{
  &\ar@{<-}[rdd]&\ar@{->}[ldd]|(.5)\bullet&\\
\ar@{->}@/_.8pc/[rrr]|(.43)\hole|(.55)\hole&&&\\  
&&&} 
\]
\caption{
\hskip 1cm 
oRIVb 
\hskip 2cm 
- 
\hskip 2cm 
oRIVa}\label{oRIVb}\label{oRIVa}
\end{figure}

\begin{figure}[H]
\[
\xymatrix@-1ex{
\ar@{->}@/^2pc/[dd]|(.14)\hole&\ar@/_2pc/[dd]|(.85)\bullet\\
&\\
&
}
\hspace{2mm}
\xymatrix{
\\
\equiv
}\hspace{2mm}
\xymatrix@-1ex{
\ar@{->}@/^2pc/[dd]|(1.16)\bullet&\ar@/_2pc/[dd]|(1.85)\hole\\
&\\
&
}
\]
\caption{$RV$}\label{Om5a}
\end{figure}

In the case of classical knot theory, coloring arcs (of projections) of 
knots ``unchanged'' by Reidemeister moves gave rise
to the definition of  quandles,  painting semiarcs led to the definition of biquandles. 
In Section 1, we generalize these definitions to color (label)
singular links/knots. The rest of the work is organized as follows:

In Section 1, after introducing the notion of singular pair, that
is the natural notion for labeling (planar diagrams of) singular knots or links,
the main result is Theorem \ref{propbialex}, that gives a full
 characterization of singular pairs when
one uses a bialexander switch $(x,y)\mapsto (sy, tx+(1-st)y)$ at classical crossings,
under the assumption that $(1-st)$ is a unit.
A consequence is  that in some cases, 
the singular switch companion to the  bialexander is necessarily a linear map
(see Theorem \ref{coro}). We end Section  1
with other types of general examples.

In Section 2 we introduce weights at crossings and the notion of non-abelian 
2-cocycle pairs. A non commutative invariant is deduced (Theorem \ref{teoinvariant})
 from the data
of a pair of cocycles $f,h$. Similarly to \cite{FG1} and \cite{FG2}, given
a set of labels $X$ and  $S,\tau:X\times X \to X\times X$ a singular pair, 
we introduce a universal group $U_{nc}^{fh}(X)$ 
factorizing  all 2-cocycles in $X$. We compute this group for some particular cases,
and the knot/link invariant that it produces. In particular, we comment 
the possible  generalizations of the notion of linking number to singular knots and links that
may be non trivial for singular knots.

In Section 3 we consider a more classical situation when the target group is abelian.
A notion of abelian cocycle pair is given
and the state sum is defined for singular knots/links. As in the non commutative case, a
  commutative group can be constructed that work as universal target for 
abelian 2-cocycle pairs and so, the state sum really depends on the choice of a singular
 pair $(S,\tau)$, and  a choice of cocycle pair is not needed, because there is always 
the universal one.

\section{Singular pairs}
First recall the set theoretical Yang-Baxter equation and the notion of 
biquandle.
\begin{defi}
A set theoretical solution of the Yang-Baxter equation is a pair
$(X,S)$ where $S:X\times X\to X\times X$ is a bijection satisfying
\[
( \id\times S)( S\times \id)( \id\times S)
=( S\times \id)( \id\times S)( S\times \id)
\]
\end{defi}
Notation:  $S(x,y)=(S^1(x,y),S^2(x,y))$ and
$S^{-1}(x,y)=\overline{S}(x,y)$. 
A solution $(X,S)$ is called
non-degenerate, or {\em birack} if in addition: 
\begin{enumerate}
 \item\label{left}
 ({\em left invertibility})
 for any $x,z\in X$ there exists a unique $y$ such that $S^1\!(x,y)=z$, 
 \item\label{right} ({\em right invertibility}) for any $y,t\in X$ there exists a unique $x$ such that $S^2(x,y)=t$.
\end{enumerate}
A birack is called {\em biquandle} if, given $x_0\in X$, there exists a unique $y_0\in X$ such that
$S(x_0,y_0)=(x_0,y_0)$. In other words, if there exists a bijective map $s:X\to X$ such
that
\[
\{(x,y):S(x,y)=(x,y)\}=
\{(x,s(x)): x\in X\}
\]
See Lemma 0.3 in \cite{FG1} for biquandle equivalent conditions.

Recall that a set $X$ with a binary operation $\t:X\times X\rightarrow X$ is 
called a {\em rack} if 
\begin{itemize}
 \item $-\t x:X\rightarrow X$ is a bijection $\forall x\in X$ and 
 \item $(x\t y)\t z=(x\t z)\t (y\t z)$ $\forall x,y,z \in X$. 
\end{itemize}
If $X$ also verifies that $x\t x=x$ then $X$ is called a {\em quandle}.  

It is clear that $(X,\t)$ is a rack if and only if
\[
S_\t(x,y):=(y,x\t y)
\]
  is a  non-degenerate set theoretical solution of the YBeq  (i.e. a birack).
Given $(X,\t)$, a rack, the birack $(X,S_\t)$ is a biquandle  if and only if $(X,\t)$ is a quandle.

\begin{defi}\label{singularpair}

Let $(X,S)$ be a biquandle and  $\tau: X\times X\to X\times X$ be a  bijective map that  
verifies  left
and right invertibility. The pair $(X, S,\tau)$ is called  {\em singular pair} 
if 
\begin{eqnarray}\tau\circ S=S\circ \tau\label{e1} \;\;\;\ \text{(coming from RV)},\\  
(S\times 1)(1\times S)(\tau\times 1)=(1\times \tau)(S\times 1)(1\times S)\label{e2} \;\;\;\ \text{(coming from RIVb)},\label{2}\\ 
(1\times S)(S\times 1)(1\times \tau)=(\tau\times 1)(1\times S)(S\times 1)\label{e3} \;\;\;\ \text{(coming from RIVa)}\label{3}
\end{eqnarray}
simultaneously holds.
\end{defi}

Equations (\ref{e2}) and (\ref{e3}) are different, take for example 
$S=\text{bialexander}(3,1,-1)$, one
can compute  the total amount of 
functions $\tau$ such that verifies (1) and (\ref{e2}), it  is different from the total amount of
 functions that verifies (1) and (\ref{e3}).
 On the other hand, if $S$ is involutive, equation (\ref{e2}) is equivalent to equation (\ref{e3}).

\begin{rem}
Given $(X, S)$ a biquandle and a bijective function $\tau$, previous definition written in elements $(x,y,z)\in X^3$ gives:
\begin{eqnarray}
\tau^1(S(x,y))=S^1(\tau(x,y)) \;\;\;\ \text{(coming from RV)}\\
\tau^2(S(x,y))=S^2(\tau(x,y)) \;\;\;\ \text{(coming from RV)}\\
\tau^1\big(S^1(x,y),S^1(S^2(x,y),z)\big)=S^1\big(x,\tau^1(y,z)\big) \;\;\;\ \text{(coming from RIVa)}\\
\tau^2\big(S^1(x,y),S^1(S^2(x,y),z)\big)=S^1\big(S^2(x,\tau^1(y,z)),\tau^2(y,z)\big) \;\;\;\ \text{(coming from RIVa)}\\
S^2\big(S^2(x,y),z\big)=S^2\big(S^2(x,\tau^1(y,z)),\tau^2(y,z)\big) \;\;\;\ \text{(coming from RIVa)}\\
S^1\big(x,S^1(y,z)\big)=S^1\big(\tau^1(x,y),S^1(\tau^2(x,y),z)\big)\;\;\;\ \text{(coming from RIVb)}\\
\tau^1\big(S^2(x,S^1(y,z)),S^2(y,z)\big)=S^2\big(\tau^1(x,y),S^1(\tau^2(x,y),z)\big) \;\;\;\ \text{(coming from RIVb)}\\
\tau^2\big(S^2(x,S^1(y,z)),S^2(y,z)\big)=S^2\big(\tau^2(x,y),z\big) \;\;\;\ \text{(coming from RIVb)}.
\end{eqnarray}
\end{rem}

\begin{rem}
The conditions for $(X,S,\tau)$ to be a singular pair are precisely the 
compatibility of the set of colorings with the oriented Reidemeister moves
(RI, RII, RIII, RIVa, RIVb and RV). In other words, given $(X,S,\tau)$ a classical pair,
the number of colorings of a link (or a knot) using $(X,S,\tau)$ is an invariant
of singular links/knots.
\end{rem}

\begin{ex}
Given a biquandle $(X,S)$, then  $(X,S,S)$ and $(X,S,S^{-1})$ are 
 singular pairs.
\end{ex}
\begin{rem}  For some biquandles these are the only possible singular pairs, see
Section \ref{sectionbialex}:
Corollary \ref{coro} and the examples after it.
\end{rem}

\begin{ex}
If $X=\{0,1\}$ and $S$=flip then there are, up to isomorphism, 2 different singular pairs:
$(X,\flip,\flip)$ and  $(X,\flip,i_2)$ 
where
$i_2$ is the nontrivial (involutive) biquandle of size 2 given by $i_2(x,y)=(y-1,x+1)$, for $x,y
\in \{0,1\}\cong \Z/2\Z$.
%
\end{ex}

\begin{ex}
If $S$=flip, then equations (2) and (3) are trivially satisfied.
The remaining condition:
\[
\tau^1(y,x)=\tau^2(x,y)
\]
\end{ex}

\begin{rem}
Given a map $\tau:X\times X\to X\times X$, the left invertibility means that, 
for any $x$, the map $\tau^1(x,-):X\to X$ is bijective,
so, to give $\tau^1$ is the same as giving a list of permutations 
$\{\tau^1(x,-)\}_{x\in X}$. Similar
 considerations for right invertibility: $\tau^2(-,y):X\to X$ is bijective for any
fixed $y\in X$. If $\#X=n$, then the set of maps 
$X\times X\to X\times X$ satisfying left and right invertibiliy condition is of cardinal
 $((n!)^n)^2$, because $\tau^1$ is determined by a list of $n$ permutations, and the
 same for $\tau^2$. Nevertheless, conditions of right and left invertibility do not imply bijectivity.
If we consider examples as above, the  restriction
$\tau^1(y,x)=\tau^2(x,y) $
 means precisely that the list of permutations determining $\tau^1$ is 
{\em the same list} as the one determining $\tau^2$.
For small $n$, one can make the list of all left and right 
invertible maps
verifying $\tau^1(y,x)=\tau^2(x,y)$ (it has  $(n!)^n$ members)
and check that very  few  give a bijective $\tau$. The following table illustrates the situation:
\[
\begin{array}{||c||c|c||c|c|||}
\hline
\# X&left-right\ invertibles& isoclasses & also\ bijective & isoclasses  \\
\hline
2     &    4&3     &   2&2\\
\hline
3    &  216&44  & 24&7\\
\hline
4&331176& 14022     &3360&169\\
\hline
\end{array}
\]

\end{rem}

\subsection{Colorings}
Let $(X,S,\tau)$ be a singular pair.
Let $L$ 
 be a singular oriented link
diagram on the plane. 
A {\em coloring} of $L$ by $X$ is a rule that assigns an element of $X$ to each semi-arc of $L$, in such a way that
for every  classical crossing (figure on the left corresponds to a positive crossing and figure on the right to a negative one)
\[
\xymatrix{
x\ar@{->}[rd]|\hole&y\ar[ld]\\
z&t
}
\hskip 2cm 
\xymatrix{
z\ar@{->}[rd]&t\ar[ld]|\hole\\
x&y
}
\]
where $(z,t)=S(x,y)$ and  in case of a singular crossing

\[
\xymatrix{
x\ar@{->}[rd]|\bullet &y\ar[ld]\\
z&t
}
\] 
where $(z,t)=\tau(x,y)$.

 Call $Col_X(L)$ the set of all possible colorings of $L$ by the singular pair $(X,S,\tau)$.

\

 Using analogous methods to the ones in \cite{HN}, the authors define {\em singular semiquandle} to color flat virtual knots and {\em virtual singular semiquandles} to color
 virtual flat singular knots.  

\begin{ex} (Singquandle from \cite{CEHN})
Let $(X,S)$ be an involutive quandle (that is, $S(x,y)=(y,x\t y)$ where $(X,\t)$ is a quandle
 satisfying  $(x\t y) \t y=x$ $\forall x,y\in X$) and a map
 $\tau(x,y)=(R_1(x,y),R_2(x,y))$ with $R_1,R_2:X\times X\to X$. 
It is straightforward to check that equations \ref{e1}, \ref{e2} and 
\ref{e3} 
leave: 
	\begin{eqnarray}
		(y \t z)\t R_2(x, z) &=& (y\t x)\t R_1(x,z) \;\;\;\text{coming from RIVa} \label{eq1}\\
		R_1(x, y)          & =&  R_2(y\t x, x) \;\;\;\text{coming from RV}\label{eq2}\\
	      R_2(x,y)   &=& 	R_1(y\t x, x)\t R_2(y\t x, x) \;\;\;\text{coming from RV} \label{eq3}\\
R_1(x\t y,z)\t y&=&R_1(x,z\t y)  \;\;\;\text{coming from RIVb} \label{eq4}\\
R_2(x\t y,z)&=&R_2(x,z\t y)\t y  \;\;\;\text{coming from RIVb} \label{eq5}	
\end{eqnarray}
which is called {\em singquandle} in \cite{CEHN}.
\end{ex}
\begin{ex}\label{ejemplorepetido}  
If $(X,\t )$ is a quandle where the operation
$\t$ is a trivial, that is
$x \t y=x$ for all $x, y \in X$ (hence $S=flip$), 
then the axioms of Definition \ref{singularpair} reduce to the condition $\tau^1(x, y) =\tau^2(y,x)$. Same as in 
Example 4.3 in \cite{CEHN}.
\end{ex}

\begin{rem}
Unlike the virtual case (see \cite{FG2}), 
where for any biquandle $(X,S)$, the pair $(X,S,\flip)$
 may be used to color a virtual knot or link, it is not always true that $(X,S,\flip)$ is a singular
pair. Moreover, to impose $\tau$=flip is a serious condition on a biquandle $S$ in order to 
get $(X,S,\flip)$ a singular pair. A simple evaluation shows the following:
\end{rem}

\begin{lem} $(X,S,\tau=\flip)$ is a singular pair if and only if $S$ simultaneously satisfies:
\[\begin{array}{rcl}
S^1(x,y)&=&S^2(y,x), \\
S^1(S^2(x,y),z)&=&S^1\big(x,z\big), \\
S^2\big(S^2(x,y),z\big)&=&S^2\big(S^2(x,z),y\big), \\
S^1\big(x,S^1(y,z)\big)&=&S^1\big(y,S^1(x,z)\big),\\
S^2(y,z)&=&S^2\big(y,S^1(x,z)\big). \\
\end{array}
\]\end{lem}
In particular, if $(X,\t)$ is a quandle and $S(x,y)=(y,x\t y)$, then the first condition of  the previous lemma says
$y=y\t x$ for all $x,y$, that is true, if and only if $S=\flip$.
Nevertheless, there are nontrivial biquandle examples:
\begin{ex}
Let $s:X\to X$ be an involutive bijection, that is, a bijective map with $s^2=\id_X$. If
one defines
$S(x,y)=(sy,sx)$, then $(X,S,\flip)$ is a singular pair.
\end{ex}

\subsection{Singular pairs for  Bialexander  switch \label{sectionbialex}}

We will study singular pairs for the $S$ = bialexander switch, under the 
general assumption that $(1-st)$ is  a unit. We point out
some cases where the only possible
singular pairs are $(S,S)$ and $(S,S^{-1})$. The main result of this section 
is the following:

\begin{teo}\label{propbialex}
Let $S(x,y)=S_{s,t}(x,y)=(sy,tx+(1-st)y)$ and assume also that $(1-st)$ is a unit.
Then  $\tau:X\times X\to X\times X$ gives a singular pair $(S,\tau)$
if and only if
\[
\begin{array}{rclc}
\tau(\lambda x,\lambda y)&=&\lambda\tau(x, y)&\hbox{for } \lambda=s,t,-1\\
\tau(x,y)&=&\tau(0,y-x/s)+(x,x/s)\\
\tau(x,y)&=&\tau(x-sy,0)+(sy,y)\\
t\tau^1(0,x)&=&s\tau^2(x,0)
\end{array}
\]
holds for all $x,y\in X$.
Moreover, $\tau$ is fully determined by $\tau^1(0,x)$,  $(x\in X$).
\end{teo}

Before the proof let us consider some consequences and examples.

\begin{rem}\label{remtauphi}
Keeping notations and hypothesis as in Theorem
\ref{propbialex},  denote $\varphi:X\to X$
 the map given by 
\[
\varphi(x):=\frac{1}{s}\ \tau^1(0,x)
\]
Because of the  non-degenerate assumption for $\tau$, we know $\varphi$ 
must be a bijective map.
The last equation of
Proposition \ref{propbialex} says that
$\tau^2(x,0)=t\varphi (x)$, and combining the second and third equation
of the same proposition, we get that necessarily $\tau$ is of the form
\[
\tau(x,y)=(x+ \varphi(sy-x),y-t\varphi(sy-x))=:\tau_\varphi(x,y)
\]
Finnally, condition $\tau(\lambda x,\lambda y)=\lambda(x,y)$ for $\lambda=s,t,-1$
is equivalent to $\varphi(\lambda(x))=\lambda \varphi(x)$ for $\lambda=s,t,-1$.

It is not difficult to see that a map $\tau_\varphi$ as above gives a singular pair
for the bialexander switch, under the only hypothesis that
$\varphi(\lambda x)=
\lambda \varphi( x)$ for $\lambda=s,t,-1$.
The proposition  above shows that when $(1-st)$ is a unit, every $\tau$ is
of the form $\tau_\varphi$.
Also, even in the case when $(1-st)$ is not necessarily
a unit, it is not hard from the  computational point of view to find
all possible bijections $\varphi$ commuting with multiplication by 
$\lambda$ ($\lambda=s,t,-1$),
and  then check if $\tau_\varphi$ is bijective or not. In other words,  this formula gives a
way to find a big family of singular pairs.
\end{rem}

\begin{teo}\label{coro}
Let $X=K$ be a finite field,  $s,t\in K^ \times$
with $t\neq s^ {-1}$ and  assume that $\{ -1,s,t\}$ generate
$K^ \times$ as multiplicative group. If $S$ is the bialexander switch associated to $s$ and 
$t$ and $\tau$ gives a singular pair for $S$ then
\begin{itemize}
\item $\tau:K^2\to K^2$ is necessarily a linear map.
\item
Denoting $a:=\tau^ 1(0,1) \in K^ \times$, we have
\[\tau^ 1(0,y)=ay,\  \tau^ 2(sx,0)=atx\]
and
\[
\tau(x,y)=
\tau_a(x,y)
=\Big(ay+\big(1-\frac{a}{s}\big)x\ , \ \frac{at}{s}x+(1-at)y \Big)
\]
\item The map $\tau_a$ is bijective if and only if $(st+1)a\neq s$.
In particular,
the map $\tau$ is necessarily  linear  and 
there are at most $|K|-1$ singular pairs for  $S_{s,t}$.
\end{itemize}
\end{teo}
\begin{proof}(of Theorem \ref{coro}, using Proposition \ref{propbialex}).
Since $\tau(  0,\lambda x)=\lambda\tau(0,x)$ for $\lambda=s,t,-1$ and
 $\langle -1,s,t\rangle=K^ \times$ if follows that
 $\tau( 0,\lambda x)=\lambda\tau(0, x)$ for $\lambda\in K^\times$. But
also $\tau(0,-x)=
-\tau(0,-x)$ implies $\tau(0,0)=(0,0)$, so 
$\tau(0,x)=x\tau(0,1)$ for all $x\in K$. 
That is, $\tau(0,-):K\to K\times K$ is $K$-linear, 
so, by the second condition of Proposition \ref{propbialex},
 $\tau$ is $K$-linear: it is given by a matrix.
Denoting $(a,b):=\tau(0,1)$,  by the second condition of 
\ref{propbialex}
we have
\[
\tau(x,y)=\tau(0,y-x/s)+(x,x/s)=
(y-x/s)(a,b)+(x,x/s)=
\Big((1-\frac{a}{s})x+ay,
\frac{1-b}{s}x+by\Big)
\]
The condition
$t\tau^1(0,x)=s\tau^2(x,0)$, for $x=1$ gives
$ta=1-b$.

 Since $\tau$ is a linear map, the bijectivity condition is controlled by
 the determinant:
\[
\det\left(
\begin{array}{cc}
1-\frac{a}{s}&a\\
\frac{at}{s}& 1-at
\end{array}\right)=1-\frac{a}{s}-at=\frac{s-a-ast}{s}=\frac{s-a(1+st)}{s}
\]
We see that it is different from zero if and only if $s\neq a(1+st)$.
\end{proof}

\begin{rem}
  $a=s$ corresponds to $\tau_a=S$ and $a=\frac{1}{t}$ corresponds to $\tau_a=S^{-1}$.
But also if $\tau$ is of the form $\tau(x,y)=(ay+bx,cx+dy)$ and $\tau$ satisfies YBeq,
then necessarily $b=0$ or $d=0$, so the only singular pairs of the form
$\tau_a$ that also satisfies YBeq are $S$ and $S^{-1}$.
\end{rem}

\begin{rem}
Let $S(x,y)=S_{s,t}(x,y)=(sy,tx+(1-st)y)$
be as in Proposition \ref{propbialex}, that is, assuming that $(1-st)$ is a unit, and assume that 
$X$ is a module over a  commutative ring containing $s$ and $t$ as  units.
If an affine map $\tau$
\[
\tau_{a,b,c,a',b',c'}(x,y)=(ax+by+c,a'x+b'y+c')
\]
(with $a,b,c,a',b',c'$ elements of the ring) gives a singular pair, then
the condition $\tau(\lambda x,\lambda y )=\lambda(x,y)$ for $\lambda=s,t,-1$ implies
$c=0=c'$, so $\tau$ is linear and the same argument
of the proof of  Corollary \ref{coro} applies, we have $\tau$ is necessarily of the form
\[
\tau(x,y)
=\Big(ay+\big(1-\frac{a}{s}\big)x\ , \ \frac{at}{s}x+(1-at)y \Big)
\]
\end{rem}

\begin{ex}
$K=\F_3=\Z/3\Z$, $s,t\in\{\pm1\}$, $s\neq t$, then the only singular pairs are $(S,S)$ and $(S,S^{-1})$,
because in this case we know there are at most $|K|-1=2$ pairs.
\end{ex}

\begin{ex}
$K=\F_4$, the field with 4 elements,  $\alpha\in \F_4$ with
$\alpha^2=\alpha+1$. Assume $s=\alpha$ and $t\neq \alpha+1=s^{-1}$.
Notice that $1+st=1-st$ is a unit, 
so the condition $(st+1)a\neq s$ is equivalent to
$a\neq \frac{s}{st+1}$. That is, $a\in \F_4\setminus\{0,\frac{s}{st+1}\}$;
we have only 2 possibilities for $a$, so again in this case the only
singular pairs are $(S,S)$ and $(S,S^{-1})$.
\end{ex}

\begin{proof}(of Proposition
\ref{propbialex})
Equations of singular pairs for $S(x,y)=(sy,tx+(1-st)y)$ are

\begin{eqnarray}
\tau^1(sy,tx+(1-st)y)&=&s\tau^2(x,y) \\
\tau^2(sy,tx+(1-st)y)&=&t\tau^1(x,y)+(1-st)\tau^2(x,y) \\
\tau^1(sy,sz)&=&s\tau^1(y,z) \label{1'}\\
\tau^2(sy,sz)&=&s \tau^2(y,z) \label{2'}\\
t(tx+(1-st)y)+(1-st)z&=&t(tx+(1-st)\tau^1(y,z))+(1-st)\tau^2(y,z) \label{3'}\\
s^2z&=&s^2z \label{s2z}\\
\tau^1(tx+(1-st)sz,ty+(1-st)z)&=&t\tau^1(x,y)+(1-st)sz \label{15}\\
\tau^2(tx+(1-st)sz,ty+(1-st)z)&=&t\tau^2(x,y)+(1-st)z .\label{16}
\end{eqnarray}
Equation \ref{s2z} is trivial. One can arrange also equations
\ref{1'} and \ref{2'} 
into a single equation
\[
\tau(sx,sy)=s\tau(x,y)
 \]
After cancelling $t^ 2x$, equation \ref{3'} becomes independent of $x$:
\[
t(1-st)y+(1-st)z=t(1-st)\tau^1(y,z)+(1-st)\tau^2(y,z) \]
but since $(1-st)$ is a unit, this is equivalent to
\[
ty+z=t\tau^1(y,z)+\tau^2(y,z) \]
Eq.
\ref{15} and \ref{16} may also be written in the form
\[\tau\big(tx+(1-st)sz,ty+(1-st)z\big)=t\tau^1(x,y)+(1-st)(sz,z)\]
Again  under the assumption $(1-st)$ being a unit, it  is equivalent to
\[\tau(tx+sz,ty+z)=t\tau(x,y)+(sz,z)\]
Notice that for $z=0$ we get $\tau(tx,ty)=t\tau(x,y)$, so it is easy to see that the above equation is equivalent to the following two:
\[\tau(x+sz,y+z)=\tau(x,y)+(sz,z)
\hbox{ and }
\tau(tx,ty)=t\tau(x,y)\]

We write again the set of equations under these simplifications:
\begin{eqnarray}
\tau^1(sy,tx+(1-st)y)&=&s\tau^2(x,y)\label{17} \\
\tau^2(sy,tx+(1-st)y)&=&t\tau^1(x,y)+(1-st)\tau^2(x,y) \label{18}\\
\tau(sx,sy)&=&s\tau(x,y) \\
tx+y&=&t\tau^1(x,y)+\tau^2(x,y) \label{20}\\
\tau(x+sz,y+z)&=&\tau(x,y)+(sz,z) \label{21}\\
\tau(tx,ty)&=&t\tau(x,y) .\label{22}
\end{eqnarray}

Notice that eq. \ref{17} says that $\tau^ 2$ is determined by $\tau^ 1$ (or vice versa).

Equation  \ref{21} for $z=-y$ or $z=-\frac{x}{s}$ gives respectively
\begin{eqnarray}
\tau(x,y)&=&\tau(x-sy,0)+(sy,y)\label{0}\\
\tau(x,y)&=&\tau(0,y-\frac{x}{s})+(x,\frac{x}{s})
\label{00}
\end{eqnarray}
So, $\tau$ is determined by $\tau(-,0)$ (and hence by $\tau^ 1(-,0)$).

Equations \ref{17} and \ref{20} for $x=0$ gives 
\begin{eqnarray}
\tau^1(sy,(1-st)y)&=&s\tau^2(0,y)\\
y&=&t\tau^1(0,y)+\tau^2(0,y)
\end{eqnarray}
and writing $x$ instead of $y$ one gets
\begin{eqnarray}
\tau^1(sx,(1-st)x)&=&s\tau^2(0,x) \label{01}\\
x&=&t\tau^1(0,x)+\tau^2(0,x)  \label{34}
\end{eqnarray}
Equation \ref{01} together with \ref{00} gives
\[
s\tau^2(0,x)=\tau^1(sx,(1-st)x)
=\tau^ 1(0,-stx)+sx
\]
hence
\[
\tau^2(0,x)=t\tau^ 1(0,-x)+x
\]
Notice that Eq. \ref{34} is $x=t\tau^1(0,x)+\tau^2(0,x)$. 
We conclude $\tau^ 1(0,-x)=-\tau^ 1(0,x)$ and consequently (use equations
\ref{0} and \ref{00})
$\tau(-x,-y)=-\tau(x,y)$.

Now  (using \ref{0} and \ref{00})  we write equations \ref{17}-\ref{22} in terms of
$\tau^ 1(0,*)$ and $\tau^2(*,0)$:
\begin{eqnarray}
\tau^1(0,tx-sty)+sy&=&s(\tau^2(x-sy,0)+y )\label{29}\\
\tau^2(-stx\!+\!s^ 2ty,\!0)\!+\!tx\!+\!(1\!-\!st)y\!&=
&\!t(\tau^1(0,y\!-\!x/s)\!+\!x   )\!+\!(1\!-\!st)(\tau^2(x-sy,\!0) \!+\!y) \label{30}\\
tx+y&=&t(\tau^1(0,y-x/s)+x)+\tau^2(x-sy,0)+y \label{31}
\end{eqnarray}
In equation \ref{30} on can simplify $tx$ and $(1-st)y$. There are also easy 
simplifications in equations \ref{29} and \ref{31}. Using also 
 that $\tau$ is $s$ and $t$ homogeneous, we get
\begin{eqnarray}
\tau^1(0,tx-sty)&=&s\tau^2(x-sy,0)\\
st\tau^2(-x + sy, 0) &=& t\tau^1(0,y - x/s) + (1 - st)\tau^2(x-sy, 0)\label{35} \\
0&=&t\tau^1(0,y-x/s)+\tau^2(x-sy,0) \label{36}
\end{eqnarray}
Equations \ref{35} and \ref{36} are actually the same equation, so we only have
\begin{eqnarray}
t\tau^1(0,x-sy)&=&s\tau^2(x-sy,0)\\
0&=&t\tau^1(0,y-x/s)+\tau^2(x-sy,0)\\
\end{eqnarray}
or equivalently
\begin{eqnarray}
t\tau^1(0,x)&=&s\tau^2(x,0)\label{1}\\
t\tau^1(0,-x)&=&-s\tau^2(x,0)
\end{eqnarray}
But in the presence of $\tau(-x,-y)=-\tau(x,y)$ this is also the same condition.
We finally conclude that, in the presence of equations \ref{0} and \ref{00}, together with 
$\tau(\lambda x,\lambda y)=\lambda\tau(x,y)$ ($\lambda=-1,s,t$), the full set of equations 
for singular pairs
is equivalent to the single  equation \ref{1}, and the claim of the proposition follows.
\end{proof}

\subsection{Computer examples}

A computer can give (just by brute force checking) 
all singular pairs $(X,S,\tau)$  for a given biquandle $(X,S)$ 
of cardinal 2 or 3, and with some time and memory, for $n=4$.
For instance, if $X=\{1,2,3\}$ and $S=\flip$, then there are (up to isomorphism) 
7 singular pairs $(X,\flip,\tau)$. Five of them
 are such that $\tau$ satisfies YBeq, four of them are actually biquandles, 
and none of them 
are given by quandles or racks, except the trivial one.
The two solutions  that {\em do  not} satisfy  the YBeq are given by
\[\begin{array}{ccccccccc}
\tau_1:\\
(1,1)&\mapsto &(1,1)\\
(2,2)&\mapsto &(3,3)&\mapsto&(2,2)\\
(1,2)&\mapsto&(3,1)&\mapsto &(3,2)&\mapsto&(1,2)\\
(1,3)&\mapsto&(2,3)&\mapsto &(2,1)&\mapsto&(1,3)\\
\\
\tau_2:\\
(1,1)&\mapsto &(2,2)&\mapsto &(3,3)&\mapsto&(1,1)\\
(1,2)&\mapsto&(1,2)\\
(1,3)&\mapsto&(3,2)&\mapsto &(3,1)&\mapsto&(2,3)&\mapsto(1,3)\\
(2,1)&\mapsto&(2,1)\\
\end{array}
\]

Given  a general biquandle $(X,S)$ with small cardinality, computational algorithms can find all possible 
$\tau$ such that  $(X,S,\tau)$ is a singular pair. Using Remark \ref{remtauphi} one can find
a family of singular pairs with bigger cardinal for bi-Alexander switch,
and thanks to Theorem
\ref{coro} we  know they are  all of this type when $(1-st)$ is a unit.
 We call $\tau_\varphi$ a switch as in Remark \ref{remtauphi}.
In  \texttt{http://mate.dm.uba.ar/\~{}mfarinat/papers/GAP/singular/} one can find
several examples implemented in GAP \cite{GAP} . We include here a table with 
some of the total amounts (as usual, we denote $D_n$ the Alexander switch 
with $s=1$, $t=-1$):
\[
\begin{array}{c}
S=\flip_n\\
\begin{array}{||c||c|c||}
\hline
n&\# of \ s.\  pairs &\#isoclasses\\
\hline
2&2&2\\
\hline
3&24&7\\
\hline
4&3360&169\\
\hline
\end{array}
\end{array}
\hskip 1cm
\begin{array}{c}
I_n=\# of\  isoclasses\ of\\  s.\  pairs\  for\  D_n\ of\  type\ \tau_\varphi\\
\begin{array}{||c||c|c|c|c|c||}
\hline
n&3&4&5&6&7\\
\hline
I_n&2&4 (all\ types=10)&6&16&20\\
\hline
\hline
n&8&9&10&11&12\\
\hline
I_n&56&136&416&776&3904\\
\hline
\end{array}
\end{array}
\]

\subsection{Examples of colorings}
\begin{ex} {\em The singular trefoil} (right) and its mirror image (left), see Figure \ref{trebolsing}.

\begin{figure}[ht]
\[
\begin{array}{c}
\includegraphics[scale=0.2]{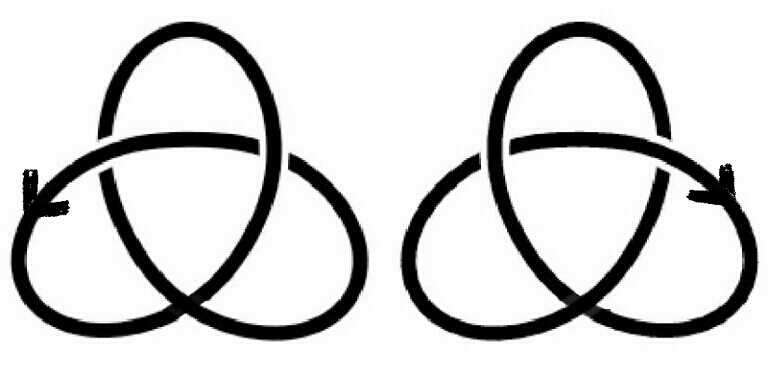}\\
\vspace{-.5in}
\includegraphics[scale=0.2]{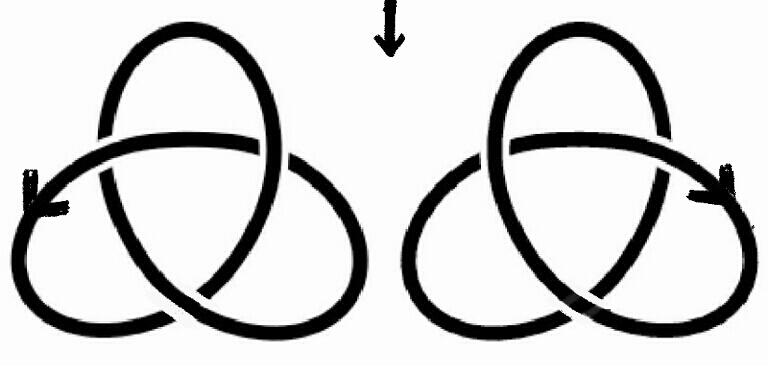}\\
\vspace{.2in}
\end{array}
\]
\caption{Singular trefoil and its mirror image} 
\label{trebolsing}
\end{figure}
Consider $(X,S)$ a biquandle and the virtual pair $p:=(S,S)$. It is clear that for the trefoil with 
negative crossings, the number of colorings equal the cardinal of the set of pairs
$\{ (x,y)\in X\times  X: S(S^{-1})^2(x,y)=(x,y)\}
=\{(x,y): (x,y)=S(x,y)\}$ while
 on the other one, the
set of colorings of the semiarcs are in 1-1  correspondence
with $\{(x,y): S^3(x,y)=(x,y)\}$.
What actually happens is that using this virtual pair, the set of colorings of the singular
knot correspond to the set of colorings of the classical knot obtained from
replacing the singular crossings by classical positive ones.
 We see  that in the first case the number of colorings is 
just the cardinal of $X$, while in the second case, the number of colorings may be different. 
For example, with the dihedral quandle of size 3 we get nontrivial colorings of the usual trefoil. 

If one uses the virtual pair $(S,S^{-1})$ then the roles of trivial or nontrivial
 colorings are interchanged.
\end{ex}

\begin{ex} {\bf The singular Hopf link}

\begin{figure}[ht]
\[
\begin{array}{c}
\includegraphics[scale=0.25]{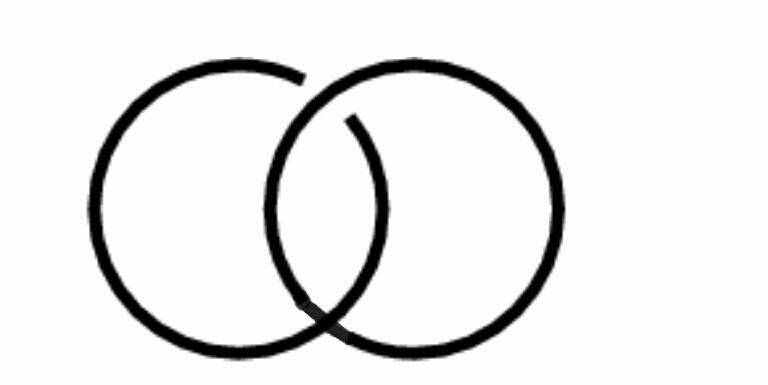}\\
\vspace{-.5in}
\end{array}
\]
\caption{Singular Hopf link} 
\label{hopfsing}
\end{figure}

Consider $X=\{1,2\}$, $S=\flip$ and $\tau=i_2$, that is, $i_2(x,y)=(y+1,x-1)$ (mod 2).
Then the set of colorings of the singular Hopf link (see Figure \ref{hopfsing}) is empty.
\end{ex}

\section{Weights \label{w} and non commutative  cocycles}

Let $(X, S, \tau)$ be a singular pair, $H$ be a group and  $f,h:X\times X\to H$ two maps. Given a coloring
on a singular knot or link, we decorate the crossing with elements of $H$ (that we call Boltzmann Weights) in the following way:
at a positive classical crossing $\gamma$, let $x_{\gamma}, y_{\gamma}$ be the color  on the incoming arcs. 
The {\it Boltzmann weight} at  $\gamma$ is 
$B_{f,h}(\gamma, \mathcal{C})=f(x_{\gamma}, y_{\gamma})$.
At a negative classical crossing $\gamma$, denote $S(x_{\gamma}, y_{\gamma})$  the colors  on the incoming  arcs.
The {\it Boltzmann weight} at $\gamma$ is 
$B_{f,h}(\gamma, \mathcal{C})=f(x_{\gamma},y_{\gamma})^{-1}$.
\[
\xymatrix@-1pc{
x_{\gamma}\ar@{->}[rdd]|\hole&y_{\gamma}\ar[ldd]\\
&\leadsto  f(x_{\gamma},y_{\gamma}	)\\
S^1\!(x_{\gamma}, y_{\gamma})&S^2(x_{\gamma},y_{\gamma})
}
\hskip 0.5cm
\xymatrix@-1pc{
S^1\!(x_{\gamma}, y _{\gamma})\ar@{->}[rdd]&S^2(x_{\gamma},y_{\gamma})\ar[ldd]|\hole\\
&\leadsto  f(x_{\gamma},y_{\gamma})^{-1}\\
x_{\gamma}&y_{\gamma}
}
\]
At a singular crossing $\gamma$, let $x_{\gamma}, y_{\gamma}$ be the color  on the incoming arcs. 
The {\it Boltzmann weight} at $\gamma$ is 
$B_{f,h}(\gamma, \mathcal{C})=h(x_{\gamma}, y_{\gamma})$.
\[
\xymatrix@-1pc{
x_{\gamma}\ar[rdd]|\bullet &y_{\gamma}\ar[ldd]\\
&\leadsto h(x_{\gamma}, y_{\gamma})&\\
\tau^1(x_{\gamma},y_{\gamma})&\tau^2(x_{\gamma},y_{\gamma})\\
}
\]

Depending on the way we choose to multiply weights,
 we arrive at two different notions of cocycles. One is
adapted for maps into general groups (e.g. non necessarily commutative) and generalizes
the non commutative cocycles for classical links given in \cite{FG1},
 and the second is only valid for maps with values
into commutative groups, its is a generalization of the state-sum procedure for biquandles 
(see \cite{CEGS}). We call respectively non-abelian 2-cocycles and abelian 2-cocycles.

\subsection{Non abelian 2-cocycle pair} 
Let $L$ be a link, $(X, S, \tau)$ a singular  pair and  $\mathcal{C}\in Col_X(L)$ be a coloring of $L$ by $X$. Call $(K_1,\dots, K_r)$ the components of $L$ and let
 $(b_1,\dots, b_r)$ be a set of base points on the components $(K_1,\dots, K_r)$.
Let $\gamma^{(i)}$, for $i=1,\dots, r$, be the (ordered) set of crossings $\gamma$ such that $\gamma$ is a  classical crossings where the under-arc
belongs to component $i$ or $\gamma$ is a virtual crossing 
of component $i$. The order of the set  $\gamma^{(i)}$ is given by the orientation of the component starting at the base point.

Notice that choosing a different base-point leads to a different product, but with the same
cyclic order, so {\em its conjugacy class} is the same.
The following definition gives the precise conditions on the maps $f$ and $h$ in 
order to make this conjugacy classes an invariant of links/knots. 

\begin{defi}\label{nc2c} The pair $(f,h)$ is called {\em non commutative 2-cocycle} if the maps $f$ and $h$ verify
\begin{itemize}
\item[(f1)] 
$f\big(x,y\big)   f\big(S^2(x,y),z\big)
 =f\big(x,S^1\!(y,z)\big) f\big(S^2(x,S^1\!(y,z)),S^2(y,z)\big) $ (due to RIII),
\item[(f2)]
$f\big(S^1\! (x, y), S^1\!(S^2(x,y),z)\big) =f\big(y,z\big) $ (due to RIII),
\item[(f3)]
$f\big(x, s(x)\big) = 1$ (due to RI),
\item[(f4)] $f(x,y)f\big(S^2(x,y),z\big)=f\big(x,\tau^1(y,z)\big)f\big(S^2(x,\tau^1(y,z)),\tau^2(y,z)\big)$ (due to RIVa),
\end{itemize}
$h$ satisfies  
\begin{itemize}
\item[(h1)] $h\big(S^1(x,y),S^1(S^2(x,y),z)\big)=h(y,z)$ (due to RIVa),
 \end{itemize}
 
 and compatibility conditions
 \begin{itemize}
 \item[(c1)] $f\big(x,S^1(y,z)\big)h\big(S^2(x,S^1(y,z)),S^2(y,z)\big)=h(x,y)f\big(\tau^2(x,y),z\big)$(due to RIVb), 
 \item[(c2)] $f(y,z)h\big(S^2(x,S^1(y,z)),S^2(y,z)\big)=h(x,y)f\big(\tau^1(x,y),S^1(\tau^2(x,y),z)\big)$ (due to RIVb).
 \item[(c3)] $h(x,y)=f(x,y)h(S(x,y))$ (due to RV), 
 \item[(c4)]  $h(S(x,y))=h(x,y)f(\tau(x,y))$ (due to RV).
 \end{itemize}
\end{defi}

\begin{rem}
If $f,h$ is a cocycle pair, then condition (c3) implies
\[
f(x,y)=h(x,y)h(S(x,y))^{-1}.
\]
In particular, $f$ is determined by $h$, so for instance if $h\equiv 1$ then $f\equiv 1$ as 
well. One may think of $h$ as a kind of ``square root'' of $f$. One may also
keep the formula $f(x,y)=h(x,y)h(S(x,y))^{-1}$ and write all others in terms 
only of $h$. For some equations  this is not particularly enlightening
and we will continue writing everything in terms of $h$ and $f$. However, there are 
two simplifications:
\end{rem}

\begin{lem}
Equation (c3) implies (f3) and equation (c3) together with  (h1) imply (f2).
\end{lem}
\begin{proof}
We know that $S(x,s(x))=(x,s(x))$, so (c3) for $(x,y)=(x,s(x))$ gives
\[
h(x,s(x))=f(x,s(x))h(x,s(x))
\]
and so $f(x,s(x))=1$. For the other implication, it is convenient to observe first that

\[
h\big(S^1(x,y),S^1(S^2(x,y),z)\big)=h(y,z)
\]
is equivalent to
\[
h\big(S\big(S^1(x,y),S^1(S^2(x,y),z)\big) \big)=h( S(y,z))
\]
This is because, for a fixed $x$, the map
\[
X\times X\to X\times X\]
\[
(y,z)\mapsto
\big(S^1(x,y),S^1(S^2(x,y),z)\big) \]
is $S$-invariant.
A diagrammatic proof is the following:
\[
\xymatrix@-2px{
x\ar@{-}[d] &&y'\ar@{-}[d]|\hole&z'\ar@{-}@/^/[dll]\\
x\ar@{-}[dr]|\hole&y\ar@{-}[dl]&z\ar@{-}@/^.5ex/[ddll]
_>>>>>>>>>>>{  S^1(S^2(x,y),z)}
|(.72)\hole\\
S^1(x,y)\ar@{-}[dr] & 
\ar@{-}[dr]&\\
S^1(x,y') &S^1(S^2(x,y'),z')&\\
}
\xymatrix@-0px{
&x\ar@/_/@{-}[dddrr]|(.36)\hole|(.73)\hole &y'\ar@{-}[dddl]&z'\ar@{-}@/^/[dddl]\\
&&&\\
\leftrightarrow
&&&&\\
&S^1(x,y') &
\hskip -4ex S^1(S^2(x,y'),z')
\hskip -4ex &\\
}
\]

\end{proof}

\begin{rem}
If $f,h$ is a cocycle pair, using conditions (c4) and (c3) we get
\[
f(\tau(x,y))=h(x,y)^{-1}f(x,y)^{-1}h(x,y)
=h(x,y)^{-1}h(S(x,y)).
\]
In particular, if $H$ is abelian, $f(\tau(x,y))=f(x,y)^{-1}$.
\end{rem}

\begin{teo}\label{teoinvariant}
Let $f,h:X\times X\to H$ be a n.c. 2-cocycle pair. For any (oriented)
diagram of a 
link $L$ using the singular pair $(X,S,\tau)$, each Reidemeister move
establish a bijection between the set of colorings of the diagram of the link L
before and after the Reidemeister move is applied. For each coloring of $L$, the (conjugacy class of the)
product of weights defined 
above is invariant under Reidemeister moves.
\end{teo}
\begin{proof}
For invariance under classical Reidemeister moves see \cite{FG1}.
For invariance under $RIVb$, we show a particular orientation (see Figure \ref{bncproof}). 
\begin{figure}[ht]
\[
  \xymatrix@-10pt{
  &x\ar@{->}@/_/[rdd]|(.22)\hole|(.47)\bullet&y\ar@{->}[ldd]|(.36)\hole&\\
S^1\!(x,\!S^1\!(y,\!z)\!)\ar@{<-}@/^1pc/[rr]&&z&\\  
&\tau^1\! (S^2\!(x,\!S^1\!(y,\!z)\!),\!S^2\!(y,\!z)\!)
&\tau^2\!(S^2\!(x,\!S^1\!(y,\!z)\!),\!S^2\!(y,\!z)\!)&\\
} 
\]
\[
 \xymatrix{
  &x\ar@{->}@/^/[rdd]|(.66)\hole&y\ar@{->}[ldd]|(.46)\bullet|(.69)\hole&\\
S^1\!(\tau^1\!(x,\!y),\!S^1\!(\tau^2\!(x,\!y),\!z)\!)\ar@{<-}@/_1.5pc/[rr]&&z&\\ 
&S^2\!(\tau^1\!(x,\!y),\!S^1\!(\tau^2\!(x,\!y),\!z)\!)&S^2\!(\tau^2\!(x,\!y),\!z)&\\ 
} 
\]
\caption{$RIVb$}
\label{bncproof}
\end{figure}

There will be no product of Boltzmann weights due to the horizontal line. When traveling the semiarc labeled by $x$ in each diagram
the product of weights will 
leave both sides of equation (c1) in 
Definition  \ref{nc2c}. When traveling the semiarc labeled by $y$ in each diagram the product 
of weights will 
leave both sides of equation (c2) in 
Definition  \ref{nc2c}.
The remaining orientations will leave equivalent equations. 

For invariance under $RIVa$, we show a particular orientation  (see Figure \ref{ancproof}). 

\begin{figure}[ht]
 \[
  \xymatrix@-6pt{
  &y\ar@{->}[rdd]&z\ar@{->}[ldd]|(.5)\bullet&\\
x\ar@{->}@/^.8pc/[rrr]|(.4)\hole|(.53)\hole&&&S^2(S^2(x,y),z)\\  
&\tau^1(S^1(x,y),S^1(S^2(x,y),z))&\tau^2(S^1(x,y),S^1(S^2(x,y),z))&\\
}
\]
\[
 \xymatrix@-6pt{
  &y\ar@{->}[rdd]&z\ar@{->}[ldd]|(.5)\bullet&\\
x\ar@{->}@/_.8pc/[rrr]|(.32)\hole|(.4)\hole&&&S^2(S^2(x,\tau^1(y,z)),\tau^2(y,z))\\  
&S^1(x,\tau^1(y,z))&S^1(S^2(x,\tau^1(y,z)),\tau^2(y,z))&} 
\]\caption{$RIVa$}
\label{ancproof}
\end{figure}

 When traveling the semiarc labeled by $x$ in each diagram
the product of weights will 
leave both sides of equation (f4) in 
Definition  \ref{nc2c}. When traveling the semiarc labeled by $y$ or $z$ in each diagram the product 
of weights will 
leave both sides of equation (h1) in 
Definition  \ref{nc2c}.
The remaining orientations will leave equivalent equations. 

\begin{figure}[ht]
\[
\xymatrix{
x\ar@/^0.5pc/[rd]|(.47)\bullet&y\ar@/_0.5pc/[ld]\\
\tau^1(x,y)\ar@/_0.5pc/[rd]|(.55)\hole&\tau^2(x,y)\ar@/^0.6pc/[ld]\\
&\\
} \hspace{2mm}
\xymatrix{
\\
\equiv
}\hspace{2mm}
\xymatrix{
x\ar@/^0.5pc/[rd]|(.43)\hole&y\ar@/_0.5pc/[ld]\\
S^1(x,y)\ar@/_0.5pc/[rd]|(.55)\bullet&S^2(x,y)\ar@/^0.6pc/[ld]\\
&\\
}
\]
\caption{RV}
\label{vncproof}
\end{figure}

For invariance under RV: when traveling the semiarc labeled by $x$,$y$ (see Figure \ref{vncproof}) in each diagram
the product of weights will 
leave both sides of equation (c3) and (c4) in 
Definition  \ref{nc2c} respectively.
The remaining orientations will leave equivalent equations. 

\end{proof}

 \subsection{Universal noncommutative 2-cocycle pair}

Given a singular pair $(X,S,\tau)$ we shall define a  group together with a
universal n.c. 2-cocycle pair in the following way:

\begin{defi}\label{def:unc}
Let $U_{nc}^{fh}=U_{nc}^{fh}(X,S,\tau)$ be the
 group freely generated by symbols
$(x,y)_f$   and $(x,y)_h$ with relations

\begin{itemize}
\item
$\big(x,y\big)_{\! f}   \big(S^2(x,y),z\big)_{\! f}
 =\big(x,S^1\!(y,z)\big)_{\! f} \big(S^2(x,S^1\!(y,z)),S^2(y,z)\big)_{\! f} $ 
\item
$(x,y)_{\! f}\big(S^2(x,y),z\big)_{\! f}=\big(x,\tau^1(y,z)\big)_{\! f} \big(S^2(x,\tau^1(y,z)),\tau^2(y,z)\big)_{\! f}$ 
\item
$\big(S^1(x,y),S^1(S^2(x,y),z)\big)_{\! h}=(y,z)_{\! h}$ 
 \item
 $\big(x,S^1(y,z)\big)_{\! f} \big(S^2(x,S^1(y,z)),S^2(y,z)\big)_{\! h}=(x,y)_{\! h} \big(\tau^2(x,y),z\big)_{\! f}$
 \item
 $(y,z)_{\! f}\big(S^2(x,S^1(y,z)),S^2(y,z)\big)_{\! h}=(x,y)_{\! h} \big(\tau^1(x,y),S^1(\tau^2(x,y),z)\big)_{\! f}$
 \item
 $(x,y)_{\! h}=(x,y)_{\! f} (S(x,y))_{\! h}$ 
 \item
 $(S(x,y))_{\! h}=(x,y)_{\! h}(\tau(x,y))_{\! f}$ 
 \end{itemize}

Denote $f_{xy}$ and $h_{xy}$ the class  in $U_{nc}^{fh}$
of $(x,y)_{\! f}$ and $(x,y)_{\! h}$ respectively. 
We also define $\pi_f,\pi_h:X\times X\to U_{nc}^{fh}$ by
\[
\pi_f,\pi_h\colon X\times X\to U_{nc}^{fh}\]\[
\pi_f(x,y):=f_{xy},\]\[
\pi_h(x,y):=h_{xy}
\]
\end{defi}

The following is immediate from the definitions:
\begin{teo}\label{teouncfg}
Let $(X,S,\tau)$ be a singular pair:
\begin{itemize}
\item   The pair of maps
$\pi_f,\pi_h\colon X\times X\to U_{nc}^{fh}$
is a noncommutative 2-cocycle pair.
\item 
Let $H$ be  a group and
 $f,h:X\times X\to H$ a  noncommutative 2-cocycle pair, then there exists a unique group homomorphism
$\rho :U_{nc}^{fh}\to H$ such that
$f=\rho \circ \pi_f$
and $h=\rho\circ \pi_h$
 \[
 \xymatrix{
 X\times X\ar[d]_{\pi_f}\ar[r]^f&H\\
 U_{nc}^{fh}\ar@{-->}[ru]_{ \rho}
 }\hskip 1cm
 \xymatrix{
 X\times X\ar[d]_{\pi_h}\ar[r]^h&H\\
 U_{nc}^{fh}\ar@{-->}[ru]_{ \rho}
 }\]
       \end{itemize}
\end{teo}

\begin{rem} $U_{nc}^{fh}$ is functorial. That is, 
 if $\phi:(X,S,\tau)\to (Y,S',\tau')$ is a morphism of 
virtual pairs, namely $\phi$
satisfy
\[
(\phi\times\phi)S(x_1,x_2)=
S'(\phi x_1,\phi x_2),
\hskip 1cm
(\phi\times\phi)\tau(x_1,x_2)=
\tau'(\phi x_1,\phi x_2)
\]
then, $\phi$ induces a (unique) group homomorphism
$U_{nc}^{fh}(X)\to U_{nc}^{fh}(Y)$
satisfying
\[
f_{x_1x_2}\mapsto
f_{\phi x_1 \phi x_2}
\hskip 1cm 
h_{x_1x_2}\mapsto
h_{\phi x_1 \phi x_2}
\]
\end{rem}

\subsection{Some particular cases}

\subsubsection{Non abelian cocycles with trivial $f$}
If $f\equiv 1$ then it is a straightforward to see that $(f,h)$ is a 2-cocycle pair if and only if
$h$ satisfies
\begin{itemize}
\item[(h1)] $h\big(S^1(x,y),S^1(S^2(x,y),z)\big)=h(y,z)$
 \item[(c1)] $h\big(S^2(x,S^1(y,z)),S^2(y,z)\big)=h(x,y)$ (c2 is the same equation), 
 \item[(c3)] $h(x,y)=h(S(x,y))$ (c4 is the same equation), 
 \end{itemize}
Notice that there is no condition involving $h$ together with $\tau$.
If in addition  $S(x,y)=(y,x\t y)$ then (h1) is trivial and the others  are
\begin{itemize}
 \item[(c1)] $h\big(x\t z,y\t z\big)=h(x,y)$ 
 \item[(c3)] $h(x,y)=h(y, x\t y)$
 \end{itemize}

\begin{ex} Take a singular link with two connected components.
Consider  $S=\tau=\flip$, $f\equiv 1$ 
and $h$ such that $h(x,y)=h(y,x)$ (as in \ref{ejemplorepetido}). Take for example $X=\{1,2\}$, so, any map $h:X\times X\to H$ with trivial $f$ (giving a 2-cocycle)
factors through 
\[
h:X\times X\to Free(a,b,c)
\]
where $h(1,1)=a$, $h(2,2)=b$, $h(1,2)=h(2,1)=c$. For a coloring with color 1 in one 
connected component and color 2 the second component, the invariant gives
products of $a$'s and $c$'s for the first component (in some cyclic order) and
 products of $b$'s and $c$'s (in a given cyclic order) for the second connected component. 
If one considers the abelianization, the exponent of $a$ gives twice the number of singular 
self intersections of the first component, the exponent of $b$ gives twice the number of
 singular self intersection of the second  component, and the exponent of $c$ gives the 
number of singular intersections of the first component with the second one. 

\end{ex}
\subsubsection{A particular virtual pair: $(X, S, S)$ with $S^2=\id$}

\begin{ex}
Take $S:X\times X\to X\times X$ an involutive (i.e. $S^2=\id$) solution
of the set theoretical Yang-Baxter equation and consider the virtual pair $(X,S,S)$. 
Denote $U:=U_{nc}^{fh}$ and $U_{ab}=U/[U,U]$ its abelianization. Then
$U_{ab}\cong Free\{(X\times X)/\sim\}_{ab}$  the free abelian group generated by 
$(X\times X)/\sim$ (written multiplicatively),  where $\sim$ is the equivalent 
relation given by
\[
(y,z)\sim (S^1(x,y),S^1(S^2(x,y),z))
\]
Moreover, $h:X\times X\to U_{ab}$ is given by
\[
h(x,y)=\ra{(x,y)}\]
and
\[f(x,y)=\ra{(x,y)}\cdot  \ra{S(x,y)}^{\ -1}\]
\end{ex}
\begin{proof}
First note that, when $S^2=\id$,  the equivalence relation
\[
(y,z)\sim (S^1(x,y),S^1(S^2(x,y),z))\ \forall x
\]
is the same as
\[
(S^2(x,S^1(y,z)),S^2(y,z))\sim (x,y)\ \forall z.\]

 Diagrammatically
 
\begin{figure}[H]
\[
\xymatrix{
x \ar@{->}@/^.5pc/[drr]|(.4)\hole|(.72)\hole&y\ar@{->}[dl]&z\ar@{->}[dl]\\
S^1(x ,y )& S^1(S^2(x,y),z)&
}
\]
\end{figure}
represents $(\id\times S)(S\times \id)(x,y,z)$.
Applying $(S\times \id)(\id\times S)$ and using twice that $S^2=\id$ leads to
\begin{figure}[H]
\[
\xymatrix{
\ar@/^15ex/[ddd]|(.16)\hole|(.36)\hole^c&\ar@/_7ex/[ddd]|(.83)\hole_a&\ar@/_15ex/[ddd]|(.64)\hole_b&
\ar@/^13ex/[ddd]|(.18)\hole|(.32)\hole|(.68)
\hole|(.82)\hole
& &\ar@/_7ex/[ddd]&\ar@/_7ex/[ddd]&
&\ar[ddd]&\ar[ddd]&\ar[ddd]&\\
&&&&&&&&&&\\
&&&\equiv&&&&\equiv&&&&\\
&S^2(a,S^1(b,c))&S^2(b,c)&&&&&&&&
}\]

\end{figure}
\
Using this, and replacing $\tau$ by $S$, we get that the cocycle conditions are
\begin{itemize}
\item[(f1)]  $f(x,y)f(S^2(x,y),z)=f(x,S^1(y,z))f(S^2(x,S^1(y,z)),S^2(y,z))$,
\item[(f4)]  $f(x,y)f(S^2(x,y),z)=f(x,S^1(y,z))f(S^2(x,S^1(y,z)),S^2(y,z))$,
\item[(h1)]  $h(S^1(x,y),S^1(S^2(x,y),z))=h(y,z)$,
\item[(c1)]  $f(x,S^1(y,z))h(x,y)=h(x,y)f(S^2(x,y),z)$,
\item[(c2)]  $f(y,z)h(x,y)=h(x,y)f(y,z)$,
\item[(c3)]  $h(x,y)=f(x,y)h(S(x,y))$,
\item[(c4)]  $h(S(x,y))=h(x,y)f(S(x,y))$
 \end{itemize}
Again from $S^2=\id$, condition  (c4) is equivalent to
\[h(x,y)=h(S(x,y))f(x,y)\]
and because we assume  the group is commutative, this is the same as (c3), 
giving the formula $f(x,y)=h(x,y)h(S(x,y))^{-1}$.

Equation (c2) is  trivial due to commutativity. Equation
(c1) gives the condition  $f(x,S^1(y,z))=f(S^2(x,y),z)$.
But the relation
\[
(y,z)\sim (S^1(x,y),S^1(S^2(x,y),z))\]
implies
\[
(x,S^1(y,z))\sim (S^2(x,y),z)
\]
So, the proof finishes by 
noticing that, if $h$ satisfy
\[
h(S^1(x,y),S^1(S^2(x,y),z))=h(y,z)
=h(S^2(x,S^1(y,z)),S^2(y,z))\]
then also
\[
h(x,S^1(y,z))=h(S^2(x,y),z)\]
and moreover, if $f(x,y)=h(x,y)h(S(x,y))^{-1}$, then $f$ also satisfies
\[
f(S^1(x,y),S^1(S^2(x,y),z))=f(y,z)
=f(S^2(x,S^1(y,z)),S^2(y,z))\]
\[
f(x,S^1(y,z))=f(S^2(x,y),z)\]
and hence (f1), (f4) and c(3) are consequences of (c3). We conclude that the value
of $h$ in $(x,y)$ depends only in the equivalence class of $(x,y)$ under the relation $\sim$,
necessarily $f(x,y)=h(x,y)h(S(x,y))^{-1}$, and no other 
 condition is needed for the pair $(f,h)$  in order to have a 2-cocycle
pair.
\end{proof}

\begin{ex}
As sub examples one can consider $X=\{1,2\}$ and $S=\tau =\flip$, in this case the relation
$\sim$ is just the identity, $U_{ab}$ is the free abelian group on 
generators $a=(1,1)$, $b=(1,2)$, $c=(2,1)$, $d=(2,2)$, 
  and the values of $f$ and $h$ are given by
\[
\begin{array}{c||c|c|c|c}
 (x,y) &(1,1)&(1,2)&(2,1)&(2,2)\\
\hline
f&1&bc^{-1}&cb^{-1}&1\\
\hline
h&a&b&c&d\\
\end{array}
 \]
If one instead considers $X=\{1,2\}$ with $S=\tau$ given by $S(x,y)=(y-1,x+1)$ Mod 2:
\[
S(1,2)=(1,2),\ S(2,1)=(2,1)
\]
\[
S(1,1)=(2,2),\ S(2,2)=(1,1)
\]
then the equivalence classes on $X\times X$ are $(1,1)\sim (2,2)$, $(1,2)\sim (2,1)$.
The group $U_{ab}$ is the free group on two generators $a=\ra{(1,1)}$ and $b=\ra{(1,2)}$.
The table with the values of $f$ and $h$ is
\[
\begin{array}{c||c|c|c|c}
 (x,y) &(1,1)&(1,2)&(2,1)&(2,2)\\
\hline
f&1&1&1&1\\
\hline
h&a&b&b&a\\
\end{array}
 \]
\end{ex}


\subsubsection{Particular examples: $S=\flip$, $\tau(x,y)=(y+1,x-1)$, $x,y\in\Z/2\Z$: extending linking number}\label{flip-antiflip}
If $S$=flip then necessarily $\tau^1(y,x)=\tau^2(x,y)$ and
 equations for the generators of $U:=U_{nc}$ are
\begin{itemize}
\item
$(x,y)_f   (x,z)_f
 =(x,z)_f (x,y)_f $ 
\item
$(x,y)_f(x,z)_f=(x,z+1)_f (x,y+1)_f$ 
 \item
 $(x,z)_f (x,y)_h=(x,y)_h (x+1,z)_f$
 \item
 $(y,z)_f(x,y)_h=(x,y)_h (y+1,z)_f$
 \item
 $(x,y)_h=(x,y)_f (y,x)_h$ 
 \item
 $(y,x)_h=(x,y)_h(y+1,x+1)_f$. 
 \end{itemize}
If we further consider the abelianization $U/[U,U]$ we get
for free the first equation; also,   from the third and fourth 
one can cancel $(x,y)_h$, getting
\begin{itemize}
\item
$(x,y)_f(x,z)_f=(x,z+1)_f (x,y+1)_f$ 
 \item
 $(x,z)_f = (x+1,z)_f$
 \item
 $(x,y)_h=(x,y)_f (y,x)_h$ 
 \item
 $(y,x)_h=(x,y)_h(y+1,x+1)_f.$ 
 \end{itemize}
From the third equation above we get
 $(x,y)_h(y,x)_ h^{-1}=(x,y)_f$. In particular $(x,x)_f=1$, but also from 
 $(x,z)_f = (x+1,z)_f$ we necessarily get $f\equiv 1$. So,
(together with $f\equiv 1$) we get the single condition
\[
(x,y)_h= (y,x)_h.
\]
We finally obtain that the abelianization of the universal group
associated to $S$=flip and $\tau=i_2$ is the free abelian group 
on 3 generators $a$, $b$, $c$,
where $f\equiv 1$ and
\[
h(1,1)=a,\ 
h(1,2)=b=h(2,1),\ 
h(2,2)=c
\]
\begin{figure}[ht]
\[
\begin{array}{c}
\includegraphics[scale=0.25]{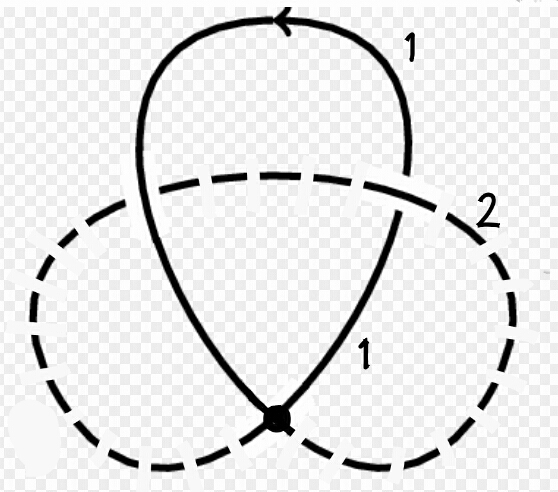}\\
\vspace{-.5in}
\end{array}
\]
\caption{Singular trefoil knot} 
\label{singtref}
\end{figure}
\begin{ex} Take the singular trefoil knot and paint it as shown in Figure \ref{singtref} with $(\Z/2\Z, S, \tau)$ as in subsection \ref{flip-antiflip}. The computation of the invariant gives $\{b^2\}$ which can be understood as a
 ``self-linking number'' similar to the case of virtual knots \cite{FG2}.
\end{ex}
\begin{ex}\label{exlinks}
Take the singular links shown in Figure \ref{nothopf}. The invariant will distinguish both links.
 Call singular hopf link the link on the right, it has  $\{cab^2,cab^2\}$ as invariant for the
 coloring shown. The invariant for  the other link, for all possible colorings, are $\{b^2,b^2\}$
 (twice) and  $\{(ca)^2,(ca)^2\}$ (twice). 
\end{ex}

\begin{figure}[ht]
\[
\begin{array}{cc}
\includegraphics[scale=0.15]{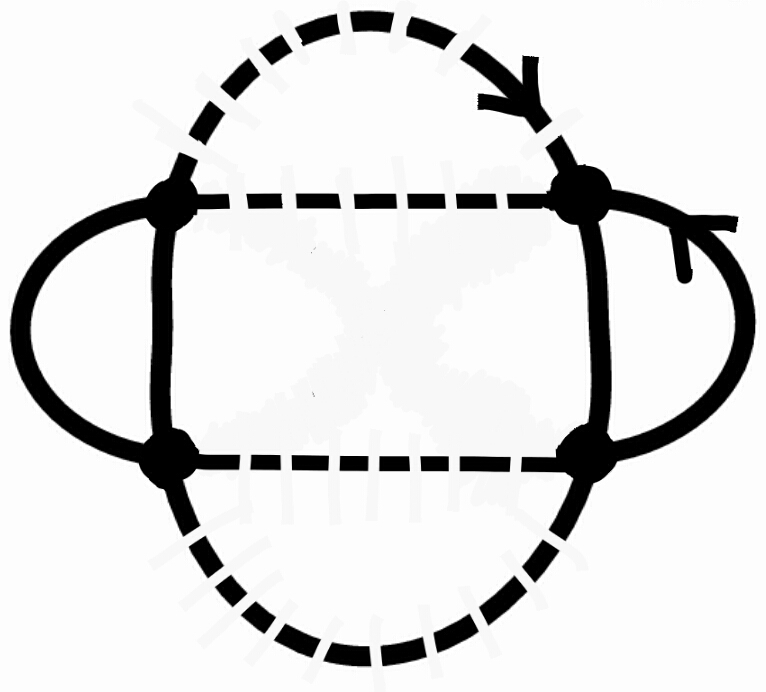}&\includegraphics[scale=0.15]{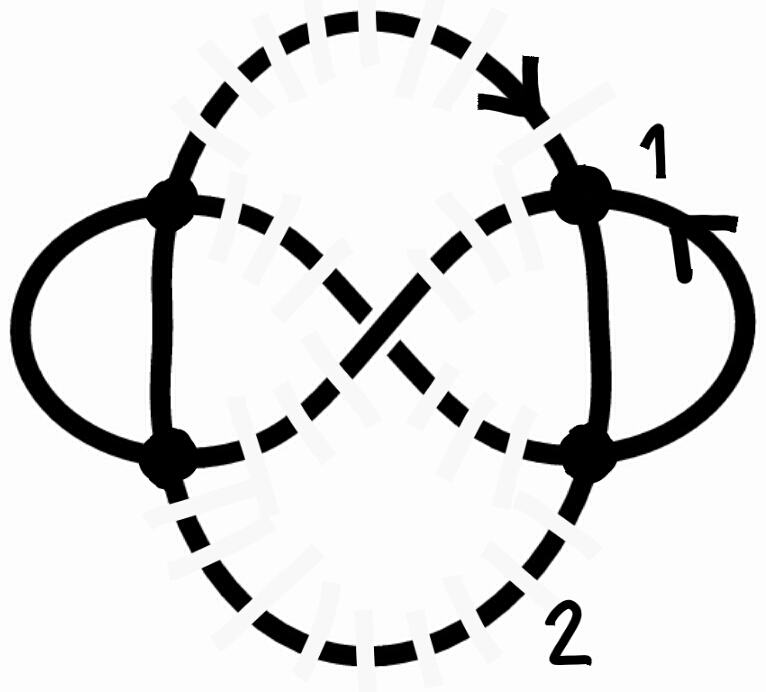}\\
\vspace{-.5in}
\end{array}
\]
\caption{4 singular crossing links} 
\label{nothopf}
\end{figure}

\section{Abelian 2-cocycle pairs and state-sum}
In this section we assume $H$ is an {\em abelian} group.  One may ignore that fact 
and consider cocycle pairs as in previous sections, or one can make 
a {\em different}  definition of 2-cocycle pair that works only in the commutative case.
We will perform this second option in this section.

 \begin{defi} 
 \label{sp} Let  $(X, S,\tau)$ be  a singular pair. 
 A pair of  functions $f,h:X\times X\rightarrow H$
is an {\em abelian  2-cocycle pair } if:
the pair $f,S$ satisfies:
\begin{itemize}
\item[(f1')] 
$ f\big(x,y\big)f\big(S^2(x,y),z\big)f\big(S^1\! (x, y), S^1\!(S^2(x,y),z)\big)\newline
 =f\big(x,S^1\!(y,z)\big)f\big(S^2(x,S^1\!(y,z)),S^2(y,z)\big)f\big(y,z\big)
$ (due to RMIII)
\item[(f2')]
$f(x, s(x))= 1$ (due to RMI) 
\end{itemize}
 and compatibility conditions between $f,h,\tau,S$:
\begin{itemize}
 \item[(c1')]  $h\big(y,z\big) f\big(x,\tau^1(y,z)\big) f\big(S^2(x,\tau^1(y,z)),\tau^2(y,z)\big)\newline
 =f\big(x,y\big) f\big(S^2(x,y),z\big) h\big(S^1(x,y), S^1(S^2(x,y),z)\big)$ (due to RIVa)
 \item[(c2')] $f(y,z) f\big(x,S^1(y,z)\big) h\big(S^2(x,S^1(y,z)),S^2(y,z)\big)\newline
 =h(x,y) f\big(\tau^2(x,y),z\big) f\big(\tau^1(x,y),S^1(\tau^2(x,y),z)\big)$ (due to RIVb)
\item[(c3')] $f(x,y) h(S(x,y))=h(x,y)f(\tau(x,y))$ (due to RV)
\end{itemize}
are satisfied for any $x, y,z \in X$.
\end{defi}   

\begin{rem}
The invariant given by non commutative cocycles were computed by traveling along
the connected components of the oriented link. Since one always meet a crossing twice, in
the non commutative case one has to choose if one considers the situation of going over or
 under arc. One of these choices gives the definition of the previous section. 
In the commutative case, one can make the simultaneous (i.e. without caring about the order)
 multiplication over all crossings, and so there is no choice of over or under arc. 
The notion of abelian cocycle pair is more symmetric, and general properties of cocycle pairs 
are very different, for instance, for non commutative cocyles, $f$ is determined by $h$,
but for abelian cocycles, $h\equiv1$ does not imply $f\equiv 1$.
\end{rem}

\begin{ex}\label{ex1}
If $(X,S)$ is a biquandle and $f:X\times X\to H$ is a biquandle (abelian) 2-cocycle
of type I, that is $f$ satisfies
\begin{itemize}
\item[(f1')] 
$ f\big(x,y\big)f\big(S^2(x,y),z\big)f\big(S^1\! (x, y), S^1\!(S^2(x,y),z)\big)\newline
 =f\big(x,S^1\!(y,z)\big)f\big(S^2(x,S^1\!(y,z)),S^2(y,z)\big)f\big(y,z\big)
$ 
\item[(f2')]
$f(x, s(x))= 1$ 
\end{itemize}
then $(X,S,S)$ is a singular pair, and $(f,h)=(f,f)$ is an abelian 2-cocycle pair. Moreover, if 
$\wt f:X\times X\to G$ is a non commutative 2-cocycle and $f$ is the composition of $f$ with the canonical
projection $G\to H:=G/[G,G]$, then $f$ is a commutative 2-cocycle, and one can construct
abelian 2-cocycles pairs in that way.
\end{ex}

\begin{ex}\label{ex2}
If $(X,S)$ is a biquandle and $f:X\times X\to H$ is an abelian 2-cocycle such that $f\circ S=f$, 
then $(X,S,S^{-1})$ is a singular pair, and $h(x,y)=f(x,y)^{-1}$ gives an abelian 2-cocycle pair.
\end{ex}

\begin{rem}
Unlike the non commutative cocycle definition, $h$ does not determine $f$. In
particular, $h\equiv 1$ does not imply $f\equiv 1$.

For instance, 
if we force $h\equiv 1$, then $f$ must satisfy
\begin{itemize}
\item[(f1')] 
$ f\big(x,y\big)f\big(S^2(x,y),z\big)f\big(S^1\! (x, y), S^1\!(S^2(x,y),z)\big)\newline
 =f\big(x,S^1\!(y,z)\big)f\big(S^2(x,S^1\!(y,z)),S^2(y,z)\big)f\big(y,z\big)$ 
\item[(f2')]
$f(x, s(x))= 1$ 
 \item[(c1')]  $ f\big(x,\tau^1(y,z)\big) f\big(S^2(x,\tau^1(y,z)),\tau^2(y,z)\big)
 =f\big(x,y\big) f\big(S^2(x,y),z\big)$
 \item[(c2')] $f(y,z) f\big(x,S^1(y,z)\big)
 = f\big(\tau^2(x,y),z\big) f\big(\tau^1(x,y),S^1(\tau^2(x,y),z)\big)$ 
\item[(c3')] $f(x,y)=f(\tau(x,y))$
\end{itemize}
\end{rem}

\begin{ex} 
For $(X,S,S)$ and $f$ satisfying (f'1) and $f\circ S=f$, then
$(f,h\equiv 1)$ is an abelian cocycle.
\end{ex}
Another special situation is the following:

\begin{ex} If  $S$=flip (recall that necessarily $\tau^1(x,y)=\tau^2(y,x)$
in this case), then 
the cocycle equations are
\begin{itemize}
\item[(f2')] $f(x, x)= 1$, 

 \item[(c1')] $ f(x,\tau^1(y,z)) f(x,\tau^2(y,z))
 =f(x,y) f(x,z)$, 

 \item[(c2')] $f(y,z) f(x,z) 
 = f(\tau^2(x,y),z) f(\tau^1(x,y),z)$, 
\item[(c3')] $f(x,y) h(y,x)=h(x,y)f(\tau(x,y))$. 
\end{itemize}
If in addition $f\equiv 1$ then the single condition for $h$ is
$h(x,y)=h(y,x)$. Also, if $\tau$=flip, then conditions (c1') and (c2') are trivial, the remaining conditions are
\[
f(x,x)=1,\ \hbox{ and }
f(x,y)h(x,y)=f(y,x)h(y,x).
\]
\end{ex}

Let $L$ be a singular link,  $(X, S,\tau)$  a singular pair, $\mathcal{C}\in Col_X(L)$ be a 
coloring of $L$ by $X$ and $f,h:X\times X\rightarrow H$
an abelian 2-cocycle pair. 
Consider the product of all Boltzmann weights associated to every crossing (classical or
 singular) in $L$:
\[
P_{\mathcal{C}}=\prod_{\gamma}B_{f,h}(\gamma,\mathcal{C})
\ \in H\]
where $B_{f,h}(\gamma,\mathcal{C})$ is  defined in the same way as in  Subsection \ref{w}.

The partition function, or state-sum (associated with $f,h$) is the expression
\[
\Phi_{f,h}(L)=\sum_{\mathcal{C}\in Col_X}P_{\mathcal{C}}
\ \in\Z[ H]
\]

\begin{teo}
The partition function defines an invariant of singular links.
\end{teo}

\begin{rem}
From a biquandle $(X,S)$ and an abelian 2-cocycle $f:X\times X\to H$ one can extend
the state-sum invariant for classical knots or links to singular ones
by considering the singular pairs and cocycle pairs as in examples \ref{ex1} and \ref{ex2}.
\end{rem}

\begin{proof} We will check the product is invariant under Reidemeister moves. 
 Same as before, we will show one orientation for each Reidemeister move, the remaining orientations will give equivalent equations. 
 For classical Reidemeister
 moves see \cite{FG1} (for quandles see  \cite{CJKLS}).
 For singular Reidemeister moves, start with oRIVa: the product of Boltzmann weights associated to the crossings shown on the top diagram of 
 Figure \ref{ancproof} is 
 $$f\big(x,y\big) f\big(S^2(x,y),z\big) h\big(S^1(x,y), S^1(S^2(x,y),z)\big)$$
 and the product of Boltzmann weights associated to the crossings shown on the bottom  diagram of Figure \ref{ancproof}
 is 
 $$h\big(y,z\big) f\big(x,\tau^1(y,z)\big) f\big(S^2(x,\tau^1(y,z)),\tau^2(y,z)\big).$$ The equality of these two elements is given by (c1') in 
 Definition \ref{sp}.
 
 The product of Boltzmann weights associated to the crossings shown on the top diagram of 
 Figure \ref{bncproof} is 
 $$f(y,z) f\big(x,S^1(y,z)\big) h\big(S^2(x,S^1(y,z)),S^2(y,z)\big)$$
 and the product of Boltzmann weights associated to the crossings shown on the bottom  diagram of Figure \ref{bncproof}
 is 
 $$h(x,y) f\big(\tau^2(x,y),z\big) f\big(\tau^1(x,y),S^1(\tau^2(x,y),z)\big).$$ The equality of these two elements is given by (c2') in 
 Definition \ref{sp}.
 
 The product of Boltzmann weights associated to the crossings shown on the left diagram of 
 Figure \ref{vncproof} is 
 $$h(x,y)f(\tau(x,y))$$
 and the product of Boltzmann weights associated to the crossings shown on the right diagram of Figure \ref{vncproof}
 is 
 $$f(x,y)h(S(x,y)).$$ The equality of these two elements is given by (c3') in 
 Definition \ref{sp}.

\end{proof}

\begin{defi}
 \label{def:abunc}
Let $Ab^{fh}=Ab^{fh}(X,S,\tau)$ be the
{\em abelian} (multiplicative) group freely generated by symbols
$(x,y)_f$   and $(x,y)_h$ with relations

\begin{itemize}
\item[(f1')] 
$ \big(x,y\big)_{\! f}\ \big(S^2(x,y),z\big)_{\! f}\ \big(S^1\! (x, y), S^1\!(S^2(x,y),z)\big)_{\! f}\newline
 =\big(x,S^1\!(y,z)\big)_{\! f}\ \big(S^2(x,S^1\!(y,z)),S^2(y,z)\big)_{\! f}\ \big(y,z\big)_{\! f}
$ 
\item[(f2')]
$(x, s(x))_{\! f}= 1$ 
 \item[(c1')]  $\big(y,z\big)_{\! h} \ \big(x,\tau^1(y,z)\big)_{\! f}\ \big(S^2(x,\tau^1(y,z)),\tau^2(y,z)\big)_{\! f}\newline
 =\big(x,y\big)_{\! f}\ \big(S^2(x,y),z\big)_{\! f}\ \big(S^1(x,y), S^1(S^2(x,y),z)\big)_{\! h}$ 
 \item[(c2')] 
$(y,z)_{\! f} \ \big(x,S^1(y,z)\big)_{\! f}\ \big(S^2(x,S^1(y,z)),S^2(y,z)\big)_{\! h}\newline
 =(x,y)_{\! h}\ \big(\tau^2(x,y),z\big)_{\! f}\ \big(\tau^1(x,y),S^1(\tau^2(x,y),z)\big)_{\! f}$
\item[(c3')] $(x,y)_{\! f} \ (S(x,y))_{\! h}=(x,y)_{\! h}\ (\tau(x,y))_{\! f}$ 
\end{itemize}
Denote $f_{xy}$ and $h_{xy}$ the class  in $Ab^{fh}$
of $(x,y)_{\! f}$ and $(x,y)_{\! h}$ respectively. 
We also define $\pi_{\! f},\pi_h:X\times X\to Ab^{fh}$ by
\[
\pi_{\! f},\pi_h\colon X\times X\to Ab^{fh}\]
\[
\pi_{\! f}(x,y):=f_{xy},
\]
\[
\pi_h(x,y):=h_{xy}
\]
\end{defi}

The following is immediate from the definitions:
\begin{teo}\label{teoucfh}
Let $(X,S,\tau)$ be a singular pair:
\begin{itemize}
\item   The pair of maps
$\pi_f,\pi_h\colon X\times X\to Ab^{fh}$
is an abelian 2-cocycle pair.
\item 
Let $H$ be  an abelian  group and
 $f,h:X\times X\to H$ an abelian 2-cocycle pair, then there exists a unique group homomorphism
$\rho :Ab^{fh}\to H$ such that
$f=\rho \circ \pi_f$
and $h=\rho\circ \pi_h$
 \[
 \xymatrix{
 X\times X\ar[d]_{\pi_f}\ar[r]^f&H\\
Ab^{fh}\ar@{-->}[ru]_{ \rho}
 }\hskip 1cm
 \xymatrix{
 X\times X\ar[d]_{\pi_h}\ar[r]^h&H\\
 Ab^{fh}\ar@{-->}[ru]_{ \rho}
 }\]
       \end{itemize}
\end{teo}

\begin{rem} $Ab^{fh}$ is functorial. That is, 
 if $\phi:(X,S,\tau)\to (Y,S',\tau')$ is a morphism of 
virtual pairs, namely $\phi$
satisfy
\[
(\phi\times\phi)S(x_1,x_2)=
S'(\phi x_1,\phi x_2),
\hskip 1cm
(\phi\times\phi)\tau(x_1,x_2)=
\tau'(\phi x_1,\phi x_2)
\]
then, $\phi$ induces a (unique) group homomorphism
$Ab^{fh}(X)\to Ab^{fh}(Y)$
satisfying
\[
f_{x_1x_2}\mapsto
f_{\phi x_1 \phi x_2}
\hskip 1cm 
h_{x_1x_2}\mapsto
h_{\phi x_1 \phi x_2}
\]
\end{rem}

\begin{rem}
In general, 
 $Ab^{fh}$ is not the abelianization of $U_{nc}^{fh}$. 
\end{rem}

\begin{question}
If $(f,h)$ is a non commutative 2-cocycle and $H$ is abelian, then $(f,h)$ is a 
commutative 2-cocycle?
In virtual case, $H$ needed to have no 2-torsion.
\end{question}

\subsubsection*{Polynomial interpretation}
An easy observation
 is that if $X$ is a {\em finite} set, then for any
singular pair on $X$, the 2-cocycle pairs always takes values into {\em finitely generated}
 abelian group, simply because the universal target group $Ab^{fh}$
is finitely generated, or because one may eventually change the target group $H$ by the 
subgroup generated by the images of the cocycles $f$ and $h$.

 If $H$ is a finitely generated abelian group, then it is necessarily 
of the form
\[
H\cong T\oplus\Z^n
\]
where $T$ is abelian and finite and $n\in \N$ is the rank of $H$.
 As a consequence, the state-sum invariant takes values into 
 a Laurent Polynomial algebra, because of the 
following isomorphisms of the group algebra:
\[
\Z[H]\cong\Z[T\oplus \Z^n] 
\cong\Z[T]\ot_\Z\Z[\Z^n]\cong \Z[T]\ot_\Z\Z[z_1^{\pm 1},\dots,z_ n^{\pm 1}]
\cong \Z[T][z_1^{\pm 1},\dots,z_ n^{\pm 1}]
\]
In this way, via the state-sum invariant, any singular pair
on a finite set $X$  associates,  to any singular knot or link,
 a Laurent polynomial,
with the same number of variables as the rank of $Ab^{fh}$.

\subsubsection*{Abelian cocycles with $S=\flip$}
We will look at abelian cocycles when $S=\flip$, in order to get generalizations of the
linking number.
For $S=\flip$, the relations in the universal group are
equivalent to
\begin{itemize}
\item[(f2')]$(x, x)_f= 1$ 
 \item[(c1')]  $ (x,\tau^1(y,z))_f (x,\tau^2(y,z))_f =(x,y)_f (x,z)_f$ 
 \item[(c2')] $(y,z)_f (x,z)_f  = (\tau^2(x,y),z)_f (\tau^1(x,y),z)_f$
\item[(c3')] $(x,y)_f (y,x)_h=(x,y)_h\tau(x,y)_f$ 
\end{itemize}
Notice that if $\tau=\flip$, then  equations (c1') and (c2') are trivially satisfied. 
Now consider the case $\tau(x,y)=(sy,sx)$ where $s:X\to X$ is a bijection
satisfying $s^2=\Id$. 
The relations are
\begin{itemize}
\item[(f2')]$(x, x)_f= 1$ 
 \item[(c1')]  $ (x,sz)_f (x,sy)_f =(x,y)_f (x,z)_f$ 
 \item[(c2')] $(y,z)_f (x,z)_f  = (sx,z)_f (sy,z)_f$
\item[(c3')] $(x,y)_f (y,x)_h=(x,y)_h(sy,sx)_f$ 
\end{itemize}
for $y=sx$, notice $\tau(x,sx)=(x,sx)$ and  (c3')  gives
\[
(sx,x)_h=(x,sx)_h
\]
Also, for $x=y=z$, (c1') and (c2') gives
\[
(sx,x)^2_f=1= (x,sx)^2_f
\]
We have almost done for the following
\begin{ex}
For $X=\{1,2\}$ and $s$ the transposition $1\leftrightarrow 2$,
the universal group $Ab^{fh}$
is given by the generators
\[
u_1=(1,2)_f,\
u_2=(2,1)_f,\
a=(1,1)_h,\
b=(1,2)_h=(2,1)_h,\
c=(2,2)_h
\]
with relations $u_1^2=1=u_2^2$ (and $(1,1)_f=1=(2,2)_f$). That is, is isomorphic to
\[
U=\Z/2\Z \times\Z/2\Z\times \Z^3
\]
 The group ring
$\Z[U]$ is isomorphic to an extension of Laurent polynomial in 3 variables:
\[
\Z[U]\cong \Z[u_1,u_2,a^{\pm 1},b^{\pm 1},c^{\pm 1}]/(u_1^2=1=u_2^2)
\]
\end{ex}

\begin{ex}
Take as above  $X=\{1,2\}$ and $s$ the transposition $1\leftrightarrow 2$.
The state sum invariant for the two singular links
shown in  
Figure \ref{nothopf} gives:
\begin{itemize}
\item
$4ab^2c$ for the link on the right,
\item $2a^2c^2+2b^4$ for the link on the left.
\end{itemize}
It worth noticing that the state-sum invariant for  $X=\{1,2\}$ but $(S,\tau)=(\flip,\flip)$
 does 
not distinguish these links.
\end{ex}


\end{document}